\documentclass{article}
\usepackage{amssymb}
\usepackage{amsfonts}
\usepackage{amsmath}

\newtheorem{theorem}{Theorem}[section]

\newtheorem{conjecture}[theorem]{Conjecture}
\newtheorem{corollary}[theorem]{Corollary}

\newtheorem{example}[theorem]{Example}

\newtheorem{lemma}[theorem]{Lemma}

\newtheorem{problem}[theorem]{Problem}

\newtheorem{remark}[theorem]{Remark}

\newenvironment{proof}[1][Proof]{\noindent\textbf{#1.} }{\ \rule{0.5em}{0.5em}}
\input{tcilatex}
\begin{document}

\title{A survey of topological Zimmer's program}
\author{Shengkui Ye}
\maketitle

\begin{abstract}
In this article, we survey the status of topological Zimmer's conjecture on
matrix group actions on manifolds.
\end{abstract}

\section{Introduction}

In mathematics, there are two important basic objects: the integers $\mathbb{%
Z}$ and the real line $\mathbb{R}$. The former is algebraic while the latter
is geometric. If we take products, we could get the finitely generated
torsion-free abelian group $\mathbb{Z}^{n}$ and the Euclidean space $\mathbb{%
R}^{k}.$ The automorphism group of $\mathbb{Z}^{n}$ is the general linear
group $\mathrm{GL}_{n}(\mathbb{Z})$ consisting of all invertible integral
matrices. Usually, we consider only the subgroup $\mathrm{SL}_{n}(\mathbb{Z}%
) $ of determinant-one matrices, which is called the special linear group.
Let $M$ be a manifold, which is a geometric object locally homeomorphic to $%
\mathbb{R}^{k}.$ We have a natural question:

\begin{problem}
How does the algebraic object $\mathrm{SL}_{n}(\mathbb{Z})$ act on the
geometric object $M?$
\end{problem}

There is an obvious action: the linear action of $\mathrm{SL}_{n}(\mathbb{Z}%
) $ on $\mathbb{R}^{n}.$ On one hand, the linear action preserves the
origin. There is an induced action of $\mathrm{SL}_{n}(\mathbb{Z})$ on the
sphere $S^{n-1}=\{x\in \mathbb{R}^{n}\mid \Vert x\Vert =1\},$ given by 
\begin{equation*}
(A,x)\longmapsto \frac{Ax}{\Vert Ax\Vert }
\end{equation*}%
for any $A\in \mathrm{SL}_{n}(\mathbb{Z})$ and $x\in S^{n-1}.$ On the other
hand, the linear action of $\mathrm{SL}_{n}(\mathbb{Z})$ preserves the
integral points $\mathbb{Z}^{n}$ in $\mathbb{R}^{n}.$ Therefore, there is
another induced action of $\mathrm{SL}_{n}(\mathbb{Z})$ on the torus $T^{n}=%
\mathbb{R}^{n}/\mathbb{Z}^{n}.$ Except for these actions, what are other
actions? The following conjecture says that when the dimension of the
manifold $M$ is smaller, any action should be a finite-group action. Let $%
\mathrm{Homeo}(M)$ (resp. $\mathrm{Diff}(M)$) be the group consisting of all
self-homeomorphisms (resp. self-diffeomorphisms) of $M.$

\begin{conjecture}
\label{conj}Any group action of $\mathrm{SL}_{n}(\mathbb{Z})$ $(n\geq 3)$ on
a connected compact manifold $M^{r}$ $(r<n-1)$ by homeomorphisms factors
through a finite group. In other words, any group homomorphism from $\mathrm{%
SL}_{n}(\mathbb{Z})$ to the homeomorphism group $\mathrm{Homeo}(M)$ has a
finite image.
\end{conjecture}

Conjecture \ref{conj} is a special case of a general conjecture in Zimmer's
program \cite{zi3}, in which the special linear group is replaced by an
arbitrary irreducible lattice $\Gamma $ in a semisimple Lie group $G$ of $%
\mathbb{R}$-rank at least 2, and the integer $n$ is replaced by a suitable
integer $h(G).$ The smooth version of Conjecture \ref{conj} was formulated
by Farb and Shalen \cite{fs}, which is related to the original Zimmer's
program. The topological version of Conjecture \ref{conj} has been discussed
by Weinberger \cite{sw}. These conjectures are part of a program to
generalize the Margulis Superrigidity Theorem to a non-linear context. In
general, it is difficult to prove these conjectures. A recent breakthrough
is that Brown-Fisher-Hurtado \cite{bhf} confirms Conjecture \ref{conj} for ($%
C^{2}$-)smooth actions. In this survey, we focus on topological actions.
Note that the topological actions could be very different from smooth
actions. Furthermore, not every topological manifold is smoothable and not
every topological manifold is triangulizable (see Freedman and Quinn \cite%
{FQ}, Chapter 8). For more details of Zimmer's program and related topics,
see survey articles of Zimmer and Morris \cite{zm}, Fisher \cite{fi} and
Labourie \cite{la}.

\section{Understand the conjecture}

First of all, the inequality in Conjecture \ref{conj} is sharp since $%
\mathrm{SL}_{n}(\mathbb{Z})$ acts non-trivially on the sphere $S^{n-1}$ when 
$n\geq 2$ (i.e. the group homomorphism $\mathrm{SL}_{n}(\mathbb{Z}%
)\rightarrow \mathrm{Homeo}(S^{n-1})$ induced by the linear transformations
has infinite image).

\subsection{Why is it for compact manifolds?}

The following result is well-known (eg. see \cite{ker}).

\begin{lemma}
Let $G=\langle g_{1},g_{2},\cdots ,g_{n}\mid r_{1},r_{2},\cdots
,r_{m}\rangle $ be a finitely presented group. There exists a closed
4-dimensional manifold $M$ with the fundamental group $\pi _{1}(M)\cong G.$
\end{lemma}

\begin{proof}
The manifold could be constructed explicitly. Let $X=\vee _{i=1}^{n}S^{1}$
be the wedge of $n$ circles. For each relator $r_{i},$ attach a 2-cell to $%
X. $ The resulting space is denoted by $Y,$ with $\pi _{1}(Y)\cong G.$ Let $%
Y\hookrightarrow \mathbb{R}^{5}$ be an embedding with a tubular neighborhood 
$N.$ The boundary $\partial N$ is a smooth closed 4-dimensional manifold
with fundamental group $G.$
\end{proof}

\bigskip

It follows that the group $G$ will act freely on the universal cover $\tilde{%
M},$ which is still a 4-dimensional manifold. When $n\geq 3,$ the special
linear group $\mathrm{SL}_{n}(\mathbb{Z})$ has a finite presentation (see
for example Milnor \cite{mi}, Corollary 10.3, p.81). Therefore, the group $%
\mathrm{SL}_{n}(\mathbb{Z})$ can act freely on a non-compact 4-dimensional
manifold for any $n\geq 3$. This means that we have to assume the manifold $%
M $ is compact in Conjecture \ref{conj}. However, we can still prove some
results for non-compact manifolds with vanishing Euler characteristics (see
Theorem \ref{bvacy}, Theorem \ref{ye}).

\subsection{Why is it finite, not trivial?}

Belolipetsky and Lubotzky \cite{bl} show that for every $n\geq 2$, every
finite group $G$ is realized as the full isometry group of some compact
hyperbolic $n$-manifold. The cases $n=2$ and $n=3$ have been proven by
Greenberg \cite{1974} and Kojima \cite{koj}, respectively.

\begin{lemma}
For any $k\geq 2,$ there exists a closed manifold $M^{k}$ such that $\mathrm{%
SL}_{n}(\mathbb{Z})$ $(n\geq 3)$ acts on $M$ non-trivially.
\end{lemma}

\begin{proof}
Let $a>1$ be a positive integer. The quotient ring homomorphism $\mathbb{%
Z\rightarrow Z}/a$ induces a surjection $\mathrm{SL}_{n}(\mathbb{Z}%
)\rightarrow \mathrm{SL}_{n}(\mathbb{Z}/a).$ Note that $\mathrm{SL}_{n}(%
\mathbb{Z}/a)$ is finite. Let $M$ be a hyperbolic manifold $M$ of dimension $%
k$ whose isometry group is $\mathrm{SL}_{n}(\mathbb{Z}/a).$ Then $\mathrm{SL}%
_{n}(\mathbb{Z})$ acts on $M$ non-trivially through $\mathrm{SL}_{n}(\mathbb{%
Z}/a).$
\end{proof}

\bigskip

This means that we can only expect the finiteness, rather than the
triviality of the image $\mathrm{SL}_{n}(\mathbb{Z})\rightarrow \mathrm{Homeo%
}(M)$ in Conjecture \ref{conj}. However, it is surprising that most
confirmed cases for Conjecture \ref{conj} indeed have trivial images (see
the next section for the status).

\section{The status}

There are many results on the smooth Zimmer program (eg. see the book \cite%
{brown}). Since this survey focuses on the topological case, we consider
only the actions by homeomorphisms. The following is a rough list of the
status of Conjecture \ref{conj}:

\begin{itemize}
\item When $r=1,$ Conjecture 1.1 is already proved by Witte \cite{wi} (see
also Ghys \cite{gh} and Burger-Monod \cite{bm} for more general lattices,
and Navas \cite{nav} for the $C^{1+\alpha }$-action $(\alpha >\frac{1}{2})$
of Kazhdan groups).

\item Weinberger \cite{we} confirms the conjecture when $M=T^{r}$ is a torus.

\item Bridson and Vogtmann \cite{bv} confirm the conjecture when $M=S^{r}$
is a sphere.

\item We confirm the conjecture for product of two spheres (\cite{ye16}),
flat manifolds \cite{ye19}, nilpotent manifolds \cite{ye181} and

\item Orientable manifolds with nonzero Euler characteristics modulo $6$
(see \cite{ye182}).
\end{itemize}

The proofs of the above mentioned results (except the cases of $%
M=S^{1},T^{r} $) use torsion elements in $\mathrm{SL}_{n}(\mathbb{Z}).$
Actually, in the above cases stronger results can be proved: the group
actions in Conjecture \ref{conj} are trivial, not just finite. However, the
following question is still open:

\begin{conjecture}
\label{1.2}Any group action of a torsion-free finite-index subgroup of $%
\mathrm{SL}_{n}(\mathbb{Z})$ on a $\mathbb{R}^{r}$ $(2\leq r<n)$ by
homeomorphisms factors through a finite group.
\end{conjecture}

\subsection{Some basic facts}

The following result on the structure of $\mathrm{SL}_{n}(\mathbb{Z})$ will
be helpful in later discussions.

\begin{lemma}
\label{normal}Let $N$ be a normal subgroup $\mathrm{SL}_{n}(\mathbb{Z}%
),n\geq 3.$ Then either $N$ lies in the center of $\mathrm{SL}_{n}(\mathbb{Z}%
)$ (and hence it is finite) or the quotient group $\mathrm{SL}_{n}(\mathbb{Z}%
)/N$ is finite.
\end{lemma}

\begin{remark}
The result holds more generally for an irreducible lattice $\Gamma $ in a
connected irreducible semisimple Lie group $G$ of real rank $\geq 2,$ which
is called the Margulis-Kazhdan theorem (see \cite{z84}, Theorem 8.1.2).
\end{remark}

The following is a fact about the splitting of exact sequences of groups
(for a proof see Brown \cite{Br}, Theorem 6.6, p.105).

\begin{lemma}
\label{split}Let 
\begin{equation*}
1\rightarrow N\rightarrow G\rightarrow Q\rightarrow 1
\end{equation*}%
be a short exact sequence of groups. Denote by $Z(N)$ the center of $N.$ If
the second cohomology group $H^{2}(Q;Z(N))=0$ and $Q$ acts trivially on $N,$
then the exact sequence is splitting.
\end{lemma}

\subsection{Group actions on 1-manifolds}

In this section, we will discuss the results proved by Witte \cite{wi}.

Recall that a group $G$ acts effectively on a space $X$ if the induced group
homomorphism $G\rightarrow \mathrm{Homeo}(X)$ is injective. The following
result is well-known (cf. \cite{na}, Theorem 2.2.19, p.38).

\begin{lemma}
If a countable group $G$ acts effectively on the real line $\mathbb{R}$ by
orientation-preserving homeomorphisms, then the group $G$ is left orderable.
In other words, there is a total order $<$ on $G$ such that $gh<gh^{\prime }$
whenever $h<h^{\prime }$ for any $g\in G.$
\end{lemma}

\begin{proof}
Choose a dense countable subset $\{x_{n}\}_{n=1}^{+\infty }$ (eg. rational
numbers $\mathbb{Q}$) of $\mathbb{R}.$ For two elements $f,g\in G,$ we
define $f<g$ if $f(x_{n})<g(x_{n})$ for the first $n$ with $f(x_{n})\neq
g(x_{n})$ but $f(x_{i})=g(x_{i}),i=1,2,\cdots ,n-1.$ It's not hard to check
that this an order. Since any $h\in G$ is orientation-preserving, we have $%
hf(x_{n})<hf(x_{n}).$ This shows that the order is a left order.
\end{proof}

\bigskip

Conversely, a countable left-orderable group $G$ acts effectively on a real
line. But when $G$ is uncountable (and left-orderable), the result may not
be true (cf. Mann \cite{mann}).

\begin{lemma}
Let $\Gamma $ be a finite-index subgroup of $\mathrm{SL}_{n}(\mathbb{Z})$ $%
(n\geq 3).$ Then $\Gamma $ is not left-orderable.
\end{lemma}

\begin{proof}
Without loss of generality, we assume $n=3.$ Let $e_{ij}(k)$ denote the
matrix with ones along diagonal, entry $k$ in the $(i,j)$-th position and
zeros elsewhere. Since $\Gamma $ is finite-index, for a sufficiently large $%
k $ we assume that $\Gamma $ contains all $e_{ij}(k).$

Suppose that $\Gamma $ has an order $<,$ which is preserved by left
multiplications. We choose $a=e_{12}(\pm k),b=e_{23}(\pm k)$ and $%
c=e_{13}(\pm k^{2})$ such that $a,b,c>1$. Note that $%
[a,b]=aba^{-1}b^{-1}=c^{-1}$ or $c.$

We claim that either $a>c^{s}$ or $b>c^{s}$ for any integer $s>0$ (denoted
as $a>>c$ or $b>>c$ in the following). To show this, suppose that $%
a<c^{m_{1}}$ and $a<c^{m_{2}}$ for some $m_{1},m_{2}>0.$ Hence $%
a<c^{k},b<c^{k}$ for $k=m_{1}m_{2}$ since $c$ is positive. Let 
\begin{equation*}
d_{m}=a^{m}b^{m}(a^{-1}c^{k})^{m}(b^{-1}c^{k})^{m}>1
\end{equation*}%
for $m>0.$ Note that $d_{m}=[a^{m},b^{m}]c^{2km}=c^{-m^{2}+2km}<1$ for
sufficiently large $m$ (assuming $[a,b]=c^{-1}.$ If $[a,b]=c,$ we may swich
the positions of $a$ and $b.$). This is a contradiction. We denote $e_{ij}$
the positive elements in $\{e_{ij}(\pm k)\}.$

We see that $e_{12}>>e_{13}$ or $e_{23}>>e_{13}.$ For the first case, note
that $[e_{13},e_{32}]=e_{12}$ and similarly we have $e_{32}>>e_{12}.$ We
thus get 
\begin{equation*}
e_{13}<<e_{12}<<e_{32}<<...<<e_{13},
\end{equation*}
a contradiction. The other case could be proved similarly.
\end{proof}

\bigskip

The previous two lemmas show that any action of a finite-index subgroup on
the real line $\mathbb{R}$ factors a finite group action.

\begin{theorem}
Any action of $\mathrm{SL}_{n}(\mathbb{Z})$ $(n\geq 3)$ by homeomorphisms on 
$S^{1}$ is trivial. In particular, Conjecture \ref{conj} holds for $r=1.$
\end{theorem}

\begin{proof}[Outline of the proof]
Let $f:\mathrm{SL}_{n}(\mathbb{Z})\rightarrow \mathrm{Homeo}(S^{1})$ be a
group homomorphism. Suppose that $\ker f=1.$ By lifting the group action to
the universal cover, we have an exact sequence%
\begin{equation*}
1\rightarrow \mathbb{Z}\rightarrow G^{\ast }\rightarrow \mathrm{SL}_{n}(%
\mathbb{Z})\rightarrow 1,
\end{equation*}%
where $G^{\ast }<\mathrm{Homeo}(\mathbb{R}).$ Note that $\mathrm{GL}_{1}(%
\mathbb{Z})=\mathbb{Z}/2$ and $\mathrm{SL}_{n}(\mathbb{Z})$ is the same as
its commutator subgroup when $n\geq 3.$ This means that $\mathrm{SL}_{n}(%
\mathbb{Z})$ acts trivially on the first term $\mathbb{Z}$. Since the second
cohomology group $H^{2}(\mathrm{SL}_{n}(\mathbb{Z});\mathbb{Z})=0,$ the
exact sequence is split by Lemma \ref{split}. Therefore, the group $\mathrm{%
SL}_{n}(\mathbb{Z})$ acts effectively on the universal cover $\mathbb{R}.$
This is impossible, since a finite cyclic group acting effectively on $%
\mathbb{R}$ is of order $2$. If $\ker f$ is not trivial, the image $\func{Im}%
f$ is finite by the congruence subgroup property (see Lemma \ref{normal}).
Now a similar argument proves that $\func{Im}f$ is trivial.
\end{proof}

\subsection{Group actions on surfaces}

The study of group actions on surfaces has a long history. The following is
well-known (see Edmonds \cite{ed}, pp340--341).

\begin{lemma}
Let $\Sigma _{g}$ be a closed orientable surface. A finite subgroup $G<%
\mathrm{Homeo}_{+}(\Sigma _{g})$ (orientation-preserving homeomorphisms) is
conjugate to a subgroup $G^{\prime }$ preserving a complex structure.
\end{lemma}

By Greenberg \cite{1974}, every finite group $G$ is a subgroup of $\mathrm{%
Homeo}(\Sigma _{g})$ for some surface $\Sigma _{g}.$

\begin{example}
A finite subgroup $G<\mathrm{Homeo}(S^{2})$ is a subgroup of $O(3)$ (cf. 
\cite{12}\cite{ed}). The set of all finite subgroups of $\mathrm{Homeo}%
(T^{2})$ can also be classified. For higher-genus surface $\Sigma _{g},g>1,$
a finite subgroup $G<\mathrm{Homeo}(\Sigma _{g})$ is a quotient group of a
non-Euclidean crystallographic group by the surface group $\pi _{1}(\Sigma
_{g})$ (cf. \cite{ggz}).
\end{example}

\begin{lemma}
Suppose that Conjecture \ref{1.2} holds for $r=2.$ Any group action of a
torsion-free finite-index subgroup $\Gamma $ of $\mathrm{SL}_{n}(\mathbb{Z})$
on a surface $\Sigma _{g}$ factors through a finite group.
\end{lemma}

\begin{proof}
Let $f:\Gamma \rightarrow \mathrm{Homeo}(\Sigma _{g})$ be a group
homomorphism. If $\ker f$ is not trivial, the image $\func{Im}f$ is finite
by the congruence subgroup property (see Lemma \ref{normal}). Suppose that $%
\ker f=1.$ Let $\mathrm{Homeo}_{0}(\Sigma _{g})$ be the identity component
of $\mathrm{Homeo}(\Sigma _{g})$ and $\mathrm{Mod}(\Sigma _{g})=\mathrm{Homeo%
}(\Sigma _{g})/\mathrm{Homeo}_{0}(\Sigma _{g})$ be the mapping class group.
Since the composite $\Gamma \rightarrow \mathrm{Homeo}(\Sigma
_{g})\rightarrow \mathrm{Mod}(\Sigma _{g})$ has finite image (by Farb-Masur 
\cite{fs2}, Theorem 1.1), there is a finite-index subgroup $\Gamma ^{\prime
} $ of $\Gamma $ such that $f(\Gamma ^{\prime })$ lies in $\mathrm{Homeo}%
_{0}(\Sigma _{g})$ the identity component of $\mathrm{Homeo}(\Sigma _{g}).$
We thus have a lifting%
\begin{equation*}
1\rightarrow \pi _{1}(\Sigma _{g})\rightarrow G^{\ast }\rightarrow \Gamma
^{\prime }\rightarrow 1,
\end{equation*}%
for some group $G^{\ast }<\mathrm{Homeo}(\mathbb{H}^{2}),$ where $\mathbb{H}%
^{2}$ is the universal cover of $\Sigma _{g}.$ Note that $\Gamma ^{\prime }$
acts trivially on $\pi _{1}(\Sigma _{g})$ and $\pi _{1}(\Sigma _{g})$ is
centerless. Thus the second cohomology group $H^{2}(\Gamma ^{\prime };Z(\pi
_{1}(\Sigma _{g})))=0$ and the above exact sequence is split by Lemma \ref%
{split}. Therefore, the group $\Gamma ^{\prime }$ could be lifted to be a
subgroup of $G^{\ast }$ and thus acts on $\mathbb{H}^{2}.$ However, the
confirmation of Conjecture \ref{1.2} for $r=2$ implies the group action of $%
\Gamma ^{\prime }$ on $\mathbb{H}^{2}$ (which is homeomorphic to the plane $%
\mathbb{R}^{2}$) is finite. This is a contradiction.
\end{proof}

\subsection{Group actions on spheres}

In the study of group actions on high dimensional manifolds, there are
several interesting phenomena. One of them is that the fixed point set of a
finite group $G$ acting on a topological manifold $M$ is not necessary a
topological manifold. Another one is that a cyclic group $\mathbb{Z}/pq$ ($%
p,q$ are two different primes) could act without fixed points on some
Euclidean space $\mathbb{R}^{n}$ (cf. Bredon \cite{Bre}, Section I.8, p.
55). A good way to deal with these difficulties is to work in the category
of (co)homology manifolds. Roughly speaking, a cohomology $n$-manifold mod $p
$ is a locally compact Hausdorff space that has a local cohomology structure
(with coefficient group $\mathbb{Z}/p$) resembling that of Euclidean $n$%
-space. Let $L=\mathbb{Z}$ or $\mathbb{Z}/p.$ All homology groups in this
section are Borel-Moore homology groups with compact supports and
coefficients in a sheaf $\mathcal{A}$ of modules over $L$. The homology
groups of $X$ are denoted by $H_{\ast }^{c}(X;\mathcal{A})$ and the
Alexander-Spanier cohomology groups (with coefficients in $L$ and compact
supports) are denoted by $H_{c}^{\ast }(X;L).$ We define the cohomology
dimension $\dim _{L}X=\min \{n\mid H_{c}^{n+1}(U;L)=0$ for all open $%
U\subset X\}.$ If $L=\mathbb{Z}/p,$ we write $\dim _{p}X$ for $\dim _{L}X.$
For integer $k\geq 0,$ let $\mathcal{O}_{k}$ denote the sheaf associated to
the pre-sheaf $U\longmapsto H_{k}^{c}(X,X\backslash U;L).$ An $n$%
-dimensional homology manifold over $L$ (denoted $n$-hm$_{L}$) is a locally
compact Hausdorff space $X$ with $\dim _{L}X<+\infty $, and $\mathcal{O}%
_{k}(X;L)=0$ for $p\neq n$ and $\mathcal{O}_{n}(X;L)$ is locally constant
with stalks isomorphic to $L$. The sheaf $\mathcal{O}_{n}$ is called the
orientation sheaf. There is a similar notion of cohomology manifold over $L$%
, denoted $n$-cm$_{L}$ (cf. \cite{Bo}, p.373). Topological manifolds are
(co)homology manifolds over $L.$

The following is called the Smith theory, which is a combination of
Corollary 19.8 and Corollary 19.9 (page 144) in \cite{Bo} (see also Theorem
4.5 in \cite{bv}).

\begin{lemma}
\label{locsmith}Let $p$ be a prime and $X$ be a locally compact Hausdorff
space of finite dimension over $\mathbb{Z}_{p}.$ Suppose that $\mathbb{Z}%
_{p} $ acts on $X$ with fixed-point set $F.$

\begin{enumerate}
\item[(i)] If $H_{\ast }^{c}(X;\mathbb{Z}_{p})\cong H_{\ast }^{c}(S^{m};%
\mathbb{Z}_{p}),$ then $H_{\ast }^{c}(F;\mathbb{Z}_{p})\cong H_{\ast
}^{c}(S^{r};\mathbb{Z}_{p})$ for some $r$ with $-1\leq r\leq m.$ If $p$ is
odd, then $r-m$ is even.

\item[(ii)] If $X$ is $\mathbb{Z}_{p}$-acyclic, then $F$ is $\mathbb{Z}_{p}$%
-acyclic (in particular non-empty and connected).
\end{enumerate}
\end{lemma}

\begin{lemma}
(\cite{bv}, Theorem 4.7) If $m<d-1$, the group $(\mathbb{Z}/2)^{d}$ cannot
act effectively on a generalized $m$-sphere over $\mathbb{Z}/2$ or a $%
\mathbb{Z}/2$-acyclic $(m+1)$-dimensional homology manifold over $\mathbb{Z}%
/2$.

If $m<2d-1$ and $p$ is odd, then $(\mathbb{Z}/p)^{d}$ cannot act effectively
a generalized $m$-sphere or a $\mathbb{Z}/p$-acyclic $(m+1)$-dimensional
homology manifold over $\mathbb{Z}/p$.
\end{lemma}

\begin{theorem}
(\cite{bv}, Theorem 1.1) Any action of $\mathrm{SL}_{n}(\mathbb{Z})$ $(n\geq
3)$ on the sphere $S^{k}$ $(k<n-1)$ by homeomorphisms is trivial.
\end{theorem}

\begin{proof}[Outline of the proof]
Note that the matrix $A=%
\begin{pmatrix}
0 & -1 \\ 
1 & -1%
\end{pmatrix}%
$ is of order $3.$ When $n$ is even, $\mathrm{SL}_{n}(\mathbb{Z})$ contains $%
\frac{n}{2}$ copies of $\mathbb{Z}/3$ generated by $A$ along the diagonal.
The previous lemma implies that this subgroup $(\mathbb{Z}/3)^{\frac{n}{2}}$
cannot act effectively on $S^{k}.$ The normal subgroup generated by this
non-trivial element acting trivially is the whole $\mathrm{SL}_{n}(\mathbb{Z}%
)$ (see Lemma \ref{normal}). When $n$ is odd, an additional argument proves
that the bound for the rank of $(\mathbb{Z}/2)^{n-1}$ in the previous lemma
could be improved by one, considering the fact that the action of $\mathrm{SL%
}_{n}(\mathbb{Z})$ is orientation-preserving. Using the elementary $2$%
-subgroup along the diagonal of $\mathrm{SL}_{n}(\mathbb{Z}),$ a similar
argument finishes the proof (for more details, see \cite{bv}, Proposition
4.13).
\end{proof}

In a similar way, the following can be proved.

\begin{theorem}
\label{bvacy}(\cite{bv}, Theorem 1.2) Any action of $\mathrm{SL}_{n}(\mathbb{%
Z})$ $(n\geq 3)$ on the acyclic space $\mathbb{R}^{k}$ $(k<n)$ by
homeomorphisms is trivial.
\end{theorem}

\subsection{Manifolds with non-zero Euler characteristics}

For a group $G$ and a prime $p,$ let the $p$-rank be $\mathrm{rk}%
_{p}(G)=\sup \{k\mid (\mathbb{Z}/p)^{k}\hookrightarrow G\}$. It is possible
that $\mathrm{rk}_{p}(G)=+\infty .$ The following result relates the $p$%
-rank of the group acting effectively and the Euler characteristic of the
acted manifold.

\begin{theorem}
\label{prop}(\cite{ye182}) Let $M^{r}$ be a first countable connected
cohomology $r$-manifold over $\mathbb{Z}/p$ and $\mathrm{Homeo}(M)$ the
group of self-homeomorphisms. We adapt the convention that $p^{n}=1$ when $%
n<0.$ Then the $p$-rank satisfies%
\begin{equation*}
p^{\mathrm{rk}_{p}(\mathrm{Homeo}(M))-[\frac{r}{2}]}\mid \chi (M;\mathbb{Z}%
/p)
\end{equation*}%
when $p$ is odd and 
\begin{equation*}
2^{\mathrm{rk}_{2}(\mathrm{Homeo}(M))-r}\mid \chi (M;\mathbb{Z}/2)
\end{equation*}%
when $p=2.$ If $M^{r}$ $(r\geq 1)$ is an oriented connected cohomology $r$%
-manifold over $\mathbb{Z}$ and $\mathrm{Homeo}_{+}(M)$ is the group of
orientation-preserving self-homeomorphisms, we have 
\begin{equation*}
2^{\mathrm{rk}_{2}(\mathrm{Homeo}(M))-r+1}\mid \chi (M;\mathbb{Z}/2).
\end{equation*}
\end{theorem}

In particular, when $(\mathbb{Z}/p)^{k}$ acts effectively on a manifold $M,$
we have that $p^{k-[\frac{r}{2}]}\mid \chi (M;\mathbb{Z}/p)$ when $p>2$ and $%
p^{k-r}\mid \chi (M;\mathbb{Z}/2)$ when $p=2.$ Based on the previous result,
we have the following criterion for non-trivial matrix group actions.

\begin{theorem}
\label{ye}(\cite{ye182}) Let $M^{r}$ be a connected (resp. orientable)
manifold with the Euler characteristic $\chi (M)\not\equiv 0\func{mod}3$
(resp. $\chi (M)\not\equiv 0\func{mod}6$). Then any group action of $\mathrm{%
SL}_{n}(\mathbb{Z})$ $(n>r+1$) on $M^{r}$ by homeomorphisms is trivial.
\end{theorem}

\begin{proof}[Outline of the proof]
Suppose that $\mathrm{SL}_{n}(\mathbb{Z})$ acts effectively on $M.$ Since
the matrix $A=%
\begin{pmatrix}
0 & -1 \\ 
1 & -1%
\end{pmatrix}%
$ is of order $3,$ the group $\mathrm{SL}_{n}(\mathbb{Z})$ contains $(%
\mathbb{Z}/3)^{[n/2]}$ as a subgroup. The previous theorem implies that $%
3\mid \chi (M).$ Consider the elementary $2$-group $(\mathbb{Z}/2)^{n-1}<%
\mathrm{SL}_{n}(\mathbb{Z})$ given by the diagonal entries $\mathrm{diag}%
(\pm 1,\cdots ,\pm 1).$ When $n\geq 3$ and $M$ is orientable, the action of $%
\mathrm{SL}_{n}(\mathbb{Z})$ is orientation-preserving. The previous theorem
implies that $2\mid \chi (M)$. Therefore, we have $6\mid \chi (M).$
\end{proof}

\bigskip

Theorem \ref{bvacy} is a special case of Theorem \ref{ye}, since the Euler
characteristic of an acyclic space is one.

\subsection{Group actions on flat manifolds}

Let $\Gamma $ be a group acting freely, isometrically, properly
discontinuously and cocompactly on the Euclidean space $\mathbb{R}^{n}.$ The
quotient space $M=\mathbb{R}^{n}/\Gamma $ is called a flat manifold. A
classical result of Bieberbach implies that there is a short exact sequence
of groups%
\begin{equation*}
1\rightarrow \mathbb{Z}^{n}\rightarrow \Gamma \rightarrow \Phi \rightarrow 1,
\end{equation*}%
where $\Phi <\mathrm{GL}_{n}(\mathbb{Z})$ is called the holonomy group of $%
M. $ Topologically, a (closed) flat manifold $M$ is covered by the torus $%
T^{n}$ with $\Phi $ as the deck transformation group.

We give a sufficient and necessary condition for a finite group to act
effectively on a flat manifold:

\begin{theorem}
\label{th3}(\cite{ye19}, Theorem 1.2) A finite group $G$ acts effectively on
a closed flat manifold $M^{n}$ with the fundamental group $\pi $ and the
holonomy group $\Phi $ by homeomorphisms if and only if there is an abelian
extension 
\begin{equation*}
1\rightarrow A\rightarrow G\rightarrow Q\rightarrow 1
\end{equation*}%
such that

\begin{enumerate}
\item[(i)] $Q\cong \Phi ^{\ast }/\Phi $ for a finite subgroup $\Phi ^{\ast }<%
\mathrm{GL}_{n}(\mathbb{Z});$

\item[(ii)] there is a $(\Phi ^{\ast },Q)$-equivariant surjection $\ \alpha :%
\mathbb{Z}^{n}\twoheadrightarrow A,$ and a commutative diagram%
\begin{equation*}
\begin{array}{ccccccccc}
1 & \rightarrow & \mathbb{Z}^{n} & \rightarrow & G^{\ast } & \rightarrow & 
\Phi ^{\ast } & \rightarrow & 1 \\ 
&  & \alpha \downarrow &  & f\downarrow &  & \downarrow &  &  \\ 
1 & \rightarrow & A & \overset{i}{\rightarrow } & G & \rightarrow & Q & 
\rightarrow & 1%
\end{array}%
\end{equation*}%
with torsion-free kernel $\ker f=\pi .$ Here \ $\alpha (gx)=\bar{g}\alpha
(x) $ for any $x\in \mathbb{Z}^{n},g\in \Phi ^{\ast }$, where $\bar{g}\in Q$
acts on the abelian group $A$ through the exact sequence and $\ker (\Phi
^{\ast }\rightarrow Q)=\Phi $.
\end{enumerate}
\end{theorem}

The condition looks a bit complicated, partially because of the holonomy
group $\Phi $. For torus $M=T^{n}$ (where $\Phi $ is trivial), we have the
following simpler characterization. To the best of our knowledge, this
characterization has so far not been stated explicitly in the literature,
except possibly for low-dimensional cases (eg. $n=1,2$).

\begin{theorem}
\label{torus}(\cite{ye19}, Theorem 1.3) A finite group $G$ acts effectively
on a torus $T^{n}$ if and only if there is an abelian extension 
\begin{equation*}
1\rightarrow A\rightarrow G\rightarrow Q\rightarrow 1
\end{equation*}%
such that

\begin{enumerate}
\item[(i)] $Q<\mathrm{GL}_{n}(\mathbb{Z});$

\item[(ii)] there is a $Q$-equivariant surjection $\ \alpha :\mathbb{Z}%
^{n}\twoheadrightarrow A$ and the cohomology class representing of the
extension lies in the image $\func{Im}(H^{2}(Q;\mathbb{Z}^{n})\rightarrow
H^{2}(Q;A)).$
\end{enumerate}
\end{theorem}

\begin{example}
When $n=1,M=S^{1},$ we have that the group $Q<\mathrm{GL}_{1}(\mathbb{Z})=%
\mathbb{Z}/2$ and $A$ is a finite cyclic group. Therefore, any finite group $%
G$ acting effectively on the circle $S^{1}$ is a subgroup of a Dihedral
group.
\end{example}

\begin{example}
Let $A_{n}$ be the alternating group. The group $A_{4}$ could act
effectively on the 2-dimensional torus $T^{2}.$
\end{example}

\begin{proof}
Note that $A_{4}\cong (\mathbb{Z}/2)^{2}\rtimes \mathbb{Z}/3,$ where $(%
\mathbb{Z}/2)^{2}=\langle (14)(23),(13)(24)\rangle $ and $\mathbb{Z}%
/3=\langle (123)\rangle ,$ where $\mathbb{Z}/3$ acts on $(\mathbb{Z}/2)^{2}$
through matrix 
\begin{equation*}
\begin{pmatrix}
0 & 1 \\ 
1 & 1%
\end{pmatrix}%
\in \mathrm{GL}_{2}(\mathbb{Z}/2).
\end{equation*}%
Take $Q=\mathbb{Z}/3.$ Define $G^{\ast }=\mathbb{Z}^{2}\rtimes \mathbb{Z}/3,$
where $\mathbb{Z}/3$ acts on the free abelian group $\mathbb{Z}^{2}$ through
matrix 
\begin{equation*}
\begin{pmatrix}
0 & -1 \\ 
1 & -1%
\end{pmatrix}%
\in \mathrm{GL}_{2}(\mathbb{Z}).
\end{equation*}%
Let 
\begin{equation*}
\alpha :\mathbb{Z}^{2}\rightarrow (\mathbb{Z}/2)^{2}
\end{equation*}%
be the modulo 2 map. It is not hard to see that $\alpha $ is $Q$%
-equivariant. Moreover, the map $\alpha $ induces a map between the two
split extensions. Theorem \ref{torus} implies that $A_{4}$ acts effectively
on $T^{2}.$
\end{proof}

\bigskip

Based on this characterization of finite groups acting effectively on flat
manifolds, we confirm Conjecture \ref{conj} for these manifolds. For a ring $%
R,$ let $E_{n}(R)$ be the subgroup of the general linear group $\mathrm{GL}%
_{n}(R)$ generated by elementary matrices. When $R=\mathbb{Z},$ we have $%
E_{n}(R)=\mathrm{SL}_{n}(\mathbb{Z}).$ Denote by $EU_{n}(R,\Lambda )$ the
elementary quadratic group (for a definition, see \cite{ye19}). Denote by $%
F_{n}$ the free group of $n$ letters and $\mathrm{SAut}(F_{n})$ (or $\mathrm{%
SOut}(F_{n})$) the unique index two subgroup of the automorphism group $%
\mathrm{Aut}(F_{n})$ (or outer automorphism group $\mathrm{Out}(F_{n})=%
\mathrm{Aut}(F_{n})/\mathrm{Inn}(F_{n})$).

\begin{theorem}
(\cite{ye19}, Theorem 1.5) Let $G=E_{n}(R)$, $EU_{n}(R,\Lambda ),$ $\mathrm{%
SAut}(F_{n})$ or $\mathrm{SOut}(F_{n}),n\geq 3.$ Suppose that $M^{r}$ is a
closed flat manifold. When $r<n,$ any group action of the group $G$ on $%
M^{r} $ by homeomorphisms is trivial, i.e. is the identity homeomorphism$.$
\end{theorem}

\begin{remark}
The previous theorem does not hold for $n=2.$ The group $\mathrm{SL}_{2}(%
\mathbb{Z})$ acts non-trivially on $S^{1}$ as shown in the introduction.
\end{remark}

\subsection{Group actions on aspherical manifolds}

A manifold $M$ is aspherical if the universal cover $\tilde{M}$ is
contractible. Note that the sphere $S^{n},n>1,$ is not aspherical.

\begin{conjecture}
\label{conj2}Any group action of $\mathrm{SL}_{n}(\mathbb{Z})$ $(n\geq 3)$
on a closed aspherical $r$-manifold $M^{r}$ by homeomorphisms factors a
finite group if $r<n.$
\end{conjecture}

The natural action of $\mathrm{SL}_{n}(\mathbb{Z})$ on the torus $T^{n}$
shows that the bound of $r$ in Conjecture \ref{conj2} cannot be improved.
Note that the upper bound of the dimension $r$ in Conjecture \ref{conj2} is $%
n-1,$ while that of Conjecture \ref{conj} is $n-2.$

\begin{lemma}
\label{lemlast}(\cite{ye181}, Theorem 1.2) Let $M^{r}$ be an aspherical
manifold$.$ A group action of $\mathrm{SL}_{n}(\mathbb{Z})$ $(n\geq 3)$ on $%
M^{r}$ $(r\leq n-1)$ by homeomorphisms is trivial if and only if the induced
group homomorphism $\mathrm{SL}_{n}(\mathbb{Z})\rightarrow \mathrm{Out}(\pi
_{1}(M))$ is trivial. In particular, Conjecture \ref{conj2} holds if the set
of group homomorphisms 
\begin{equation*}
\mathrm{Hom}(\mathrm{SL}_{n}(\mathbb{Z}),\mathrm{Out}(\pi _{1}(M)))=1.
\end{equation*}
\end{lemma}

\begin{proof}[Outline of the proof]
It is clear that when the group action is trivial, the induced group
homomorphism is trivial. Conversely, suppose that $\mathrm{SL}_{n}(\mathbb{Z}%
)$ acts on $M$ effectively. We have an exact sequence%
\begin{equation*}
1\rightarrow \pi _{1}(M)\rightarrow G\rightarrow \mathrm{SL}_{n}(\mathbb{Z}%
)\rightarrow 1,
\end{equation*}%
where $G<\mathrm{Homeo}(\tilde{M}),$ the homeomorphism group of the
universal cover. Since the dimension of $M$ is smaller than $n,$ the rank of
the center $Z(\pi _{1}(M))$ is smaller than $n$ as well by considering the
cohomological dimension. This implies that the action of $\mathrm{SL}_{n}(%
\mathbb{Z})$ on $Z(\pi _{1}(M))$ is trivial. It is known that $H^{2}(\mathrm{%
SL}_{n}(\mathbb{Z});\mathbb{Z})=0.$ By Lemma \ref{split}, the exact sequence
is splitting. Then we have a group action of $\mathrm{SL}_{n}(\mathbb{Z})$
on the universal cover $\tilde{M},$ which is contractible. A result of
Bridson and Vogtmann \cite{bv} (see also Theorem \ref{bvacy}) implies that
the action of $\mathrm{SL}_{n}(\mathbb{Z})$ on the acyclic space $\tilde{M}$
(and thus on $M$) is trivial.
\end{proof}

\bigskip

The previous lemma reduces Conjecture \ref{conj2} to an algebraic question:
whether any group homomorphism from $\mathrm{SL}_{n}(\mathbb{Z})$ to $%
\mathrm{Out}(\pi _{1}(M)$ is trivial. This is the case when $\pi _{1}(M)$ is
nilpotent.

\begin{theorem}
(\cite{ye181}, Theorem 1.4) Let $M^{r}$ be an aspherical manifold$.$ If the
fundamental group $\pi _{1}(M)$ is finitely generated nilpotent, any group
action of $\mathrm{SL}_{n}(\mathbb{Z})$ $(n\geq 3)$ on $M^{r}$ $(r\leq n-1)$
by homeomorphisms is trivial.
\end{theorem}

Similarly, when $M$ is an almost flat manifold (i.e. a manifold finitely
covered by a nil-manifold) with dihedral, symmetric or alternating holonomy
group, it satisfies Conjecture \ref{conj2} (see \cite{ye181}, Lemma 6.5).

\begin{corollary}
Let $M$ be a closed aspherical manifold. Any action of the real special
linear group $\mathrm{SL}_{n}(\mathbb{R}),n\geq 3,$ on $M$ by homeomorphism
is trivial.
\end{corollary}

\begin{proof}
When $M$ is compact, the fundamental group $\pi _{1}(M)$ is finitely
presented and $\mathrm{Out}(\pi _{1}(M))$ is countable. Therefore, any group
homomorphism $\mathrm{SL}_{n}(\mathbb{R})\rightarrow \mathrm{Out}(\pi
_{1}(M))$ is trivial (note that $\mathrm{PSL}_{n}(\mathbb{R})$ is a simple
group). Lemma \ref{lemlast} implies that the action of $\mathrm{SL}_{n}(%
\mathbb{Z})$ (and thus $\mathrm{SL}_{n}(\mathbb{R}))$ is trivial, when $\dim
(M)<n$. Actually, a compact Lie group (including finite group) acting
effectively and homotopically trivially on $M$ must be abelian (see \cite%
{glo}, Theorem 2.5). The group $\mathrm{SL}_{n}(\mathbb{R})$ contains a
non-abelian finite subgroup (for example, the alternating group $A_{n+1}$).
This implies that the action of $\mathrm{SL}_{n}(\mathbb{R})$ is trivial for
any $M.$
\end{proof}

\bigskip

\noindent \textbf{Acknowledgements}

The author is grateful to Prof. Xuezhi Zhao at the Capital Normal University
and Prof. Yang Su at the Chinese Academy of Sciences for many helpful
discussions. This work is supported by NSFC (No. 11971389).

\bigskip

NYU Shanghai, 1555 Century Avenue, Shanghai, 200122, China.

NYU-ECNU Institute of Mathematical Sciences, NYU Shanghai, 3663 Zhongshan
Road North, Shanghai, 200062, China.

E-mail: sy55@nyu.edu, yeshengkui@gmail.com

\end{document}